\documentclass{amsart}

\usepackage{color}
\usepackage{amsfonts}
\usepackage{amssymb}
\usepackage{enumerate}
\usepackage{amsmath}
\usepackage{amsthm}
\usepackage{graphicx}
\usepackage[english]{babel}

\numberwithin{equation}{section}

\newtheorem{theorem}{Theorem}[section]
\newtheorem{lemma}[theorem]{Lemma}
\newtheorem{question}[theorem]{Question}
\newtheorem{corollary}[theorem]{Corollary}
\newtheorem{fact}[theorem]{Fact}

\theoremstyle{definition}
\newtheorem{definition}[theorem]{Definition}

\DeclareMathOperator{\diam}{diam}
\DeclareMathOperator{\Var}{Var}

\begin{document}

\title{Restrictions of Brownian motion}

\author{Rich\'ard Balka}
\address{Department of Mathematics, University of Washington, Box 354350, Seattle, WA 98195-4350, USA
and Alfr\'ed R\'enyi Institute of Mathematics, Hungarian Academy of Sciences, PO Box 127, 1364 Budapest, Hungary}
\email{balka@math.washington.edu}

\author{Yuval Peres}

\address{Microsoft Research, 1 Microsoft Way, Redmond, WA 98052, USA}
\email{peres@microsoft.com}



\begin{abstract} Let $\{ B(t) \colon 0\leq t\leq 1\}$ be a linear Brownian motion and let $\dim$ denote the Hausdorff dimension.
Let $\alpha>\frac12$ and $1\leq \beta \leq 2$. We prove that, almost surely, there exists no set $A\subset[0,1]$ such that $\dim A>\frac12$
and $B\colon A\to\mathbb{R}$ is $\alpha$-H\"older continuous. The proof is an application of Kaufman's dimension doubling theorem.
As a corollary of the above theorem, we show that, almost surely, there exists
no set $A\subset[0,1]$ such that $\dim A>\frac{\beta}{2}$ and $B\colon A\to\mathbb{R}$ has finite $\beta$-variation.
The zero set of $B$ and a deterministic construction witness that the above theorems give the optimal dimensions.

\bigskip

\begin{center} { \scshape{RESTRICTIONS DU MOUVEMENT BROWNIEN}} \end{center}

\bigskip

\noindent {\scshape{R\'esum\'e}}. On note $\{B(t) \colon  0 \leq t  \leq 1 \}$ un mouvement brownien lin\'eaire et $\dim$ la dimension de Hausdorff.
Pour $\alpha> \frac 12$ et $1 \leq \beta \leq 2$ nous montrons que, presque s\^urement, il n'existe pas d'ensemble $A\subset  [0,1]$ tel que
$\dim A> \frac 12$ et $B\colon A \to \mathbb{R}$ soit $\alpha$-H\"older continue. La preuve est une application du th\'eor\`eme de Kaufman sur le doublement de dimension. Comme corollaire du th\'eor\`eme ci-dessus, nous montrons que, presque s\^urement, il n'existe
pas d'ensemble $A \subset [0,1]$ tel que $\dim A> \frac{\beta}{2}$ et $B\colon A\to\mathbb{R}$ ait une $\beta$-variation finie. L'ensemble des z\'eros de $B$ et une construction d\'eterministe montrent que que les th\'eor\`emes ci-dessus donnent les dimensions optimales.
\end{abstract}

\maketitle

\section{Introduction}

We examine how  large a set can be, on
which linear Brownian motion is $\alpha$-H\"older continuous for some $\alpha>\frac12$ or
has finite $\beta$-variation for some $1\leq \beta\leq 2$.
The main goal of the paper is to prove the following two theorems.

\begin{theorem} \label{t:Holder}
Let $\{ B(t)\colon 0\leq t\leq 1\}$ be a linear Brownian motion and assume that $\alpha>\frac12$.
Then, almost surely, there exists no set~$A\subset[0,1]$ with $\dim A>\frac12$
such that $B\colon A\to\mathbb{R}$ is $\alpha$-H\"older continuous.
\end{theorem}

Recall that for $A\subset [0,1]$ the \emph{$\beta$-variation} of a function $f\colon A\to \mathbb{R}$ is defined as
$$\Var^{\beta} (f)=\sup\left\{\sum_{i=1}^{n} |f(x_{i})-f(x_{i-1})|^{\beta}: x_0<\dots <x_n,~x_i\in A,~n\in \mathbb{N}^+\right\}.$$

\begin{theorem} \label{t:BV} Let $\{ B(t)\colon 0\leq t\leq 1\}$ be a linear Brownian motion and assume that $1 \leq \beta \leq 2$.
Then, almost surely, there exists no set $A\subset[0,1]$ with $\dim A>\frac{\beta}{2}$
such that $B\colon A\to\mathbb{R}$ has finite $\beta$-variation. In particular,
$$\mathbb{P}(\exists A: \dim A> \textstyle \frac 12 \textrm{ and } B|_{A} \textrm{ is increasing})=0.$$
\end{theorem}

Clearly, the above theorems hold simultaneously for a countable dense set of parameters $\alpha,\beta$, thus
simultaneously for all $\alpha,\beta$. Let $\mathcal{Z}$ be the zero set of a linear Brownian motion $B$.
Then, almost surely, $\dim \mathcal{Z}=\frac 12$ and $B|_{\mathcal{Z}}$ is $\alpha$-H\"older continuous for all $\alpha>\frac 12$, so Theorem~\ref{t:Holder} gives the optimal dimension. We prove also that Theorem~\ref{t:BV} is best possible, see Theorem~\ref{t:opt}.

\subsection*{Motivation and related results}
Let $C[0,1]$ denote the set of continuous functions $f\colon [0,1]\to \mathbb{R}$ endowed with the maximum norm.
Elekes proved the following restriction theorem.

\begin{theorem}[Elekes \cite{E}] Let $0<\alpha<1$. For the generic continuous function $f\in C[0,1]$ (in the sense of Baire category)
\begin{enumerate}[(1)]
\item  for all $A\subset [0,1]$, if $f|_{A}$ is $\alpha$-H\"older continuous, then $\dim A\leq 1-\alpha$;
\item for  all $A\subset [0,1]$, if $f|_{A}$ is of bounded variation, then $\dim A\leq \frac 12$.
\end{enumerate}
\end{theorem}

The above theorem is sharp, the following result was proved by Kahane and Katznelson,
and M\'ath\'e independently, by different methods.

\begin{theorem}[Kahane and Katznelson \cite{KK}, M\'ath\'e \cite{Ma}] \label{t:KKM} Let $0<\alpha<1$. For any $f\in C[0,1]$ there are compact sets
$A,D\subset [0,1]$ such that
\begin{enumerate}[(1)]
\item $\dim A=1-\alpha$ and $f|_{A}$ is $\alpha$-H\"older continuous;
\item $\dim D=\frac 12$ and $f|_{D}$ is of bounded variation.
\end{enumerate}
\end{theorem}

Kahane and Katznelson also considered H\"older continuous functions.

\begin{definition} For $A\subset [0,1]$ let $C^{\alpha}(A)$ and $BV(A)$ denote the set of functions $f\colon A\to \mathbb{R}$ that are
$\alpha$-H\"older continuous and of bounded variation, respectively. For all $0<\alpha<\beta<1$ define
\begin{align*} H(\alpha,\beta)&=\sup\{\gamma: \forall \, f\in C^{\alpha}[0,1] \, \exists A\subset [0,1] \textrm{ s.t. }
\dim A=\gamma \textrm{ and } f|_{A}\in C^{\beta}(A) \}, \\
V(\alpha)&=\sup\{\gamma: \forall \, f\in C^{\alpha}[0,1] \, \exists A\subset [0,1] \textrm{ s.t. }
\dim A=\gamma \textrm{ and } f|_{A}\in BV(A)\}.
\end{align*}
\end{definition}

\begin{theorem}[Kahane and Katznelson \cite{KK}] For all $0<\alpha<\beta<1$ we have
$$H(\alpha,\beta)\leq \frac{1-\beta}{1-\alpha} \quad \textrm{and} \quad V(\alpha)\leq \frac{1}{2-\alpha}.$$
\end{theorem}

\begin{question}[Kahane and Katznelson \cite{KK}] \label{q:KK} Is the above result best possible?
\end{question}

As the linear Brownian motion $B$ is $\alpha$-H\"older continuous for all $\alpha< \frac 12$,
our results and Theorem~\ref{t:KKM} imply the following corollary.

\begin{corollary} For all $0<\alpha<\frac 12<\beta<1$ we have
$$H(\alpha,\beta)\leq \frac 12 \quad \textrm{and} \quad V(\alpha)=\frac 12.$$
\end{corollary}

Related results in the discrete setting can be found in \cite{ABP}.

\begin{definition}
Let $d\geq 2$ and $f\colon [0,1] \to \mathbb{R}^d$. We say that $f$ is \emph{increasing} on a set $A\subset [0,1]$
if all the coordinate functions of $f|_{A}$ are non-decreasing.
\end{definition}

\begin{question} Let $d\geq 2$ and let $\{ B(t)\colon 0\leq t\leq 1\}$ be a standard $d$-dimensional Brownian motion. What is the maximal
number $\gamma$ such that, almost surely, $B$ is increasing on some set of Hausdorff dimension $\gamma$?
\end{question}

\section{Preliminaries}

The diameter of a metric space $X$ is denoted by $\diam X$. For all $s \ge 0$ the \emph{$s$-dimensional Hausdorff measure} of $X$ is defined as
\begin{align*}
\mathcal{H}^{s}(X)&=\lim_{\delta\to 0+}\mathcal{H}^{s}_{\delta}(X)
\mbox{, where}\\
\mathcal{H}^{s}_{\delta}(X)&=\inf \left\{ \sum_{i=1}^\infty (\diam
X_{i})^{s}: X \subset \bigcup_{i=1}^{\infty} X_{i},~
\forall i \diam X_i \le \delta \right\}.
\end{align*}
The \emph{Hausdorff dimension} of $X$ is defined as
$$\dim X = \inf\{s \ge 0: \mathcal{H}^{s}(X) <\infty\}.$$
Let $A\subset \mathbb{R}$ and $\alpha>0$. A function $f\colon A\to \mathbb{R}$ is called
\emph{$\alpha$-H\"older continuous} if there exists a constant $c\in (0,\infty)$
such that $|f(x)-f(y)|\leq c|x-y|^{\alpha}$ for all $x,y\in A$.

\begin{fact} \label{f:H} If $f\colon A\to \mathbb{R}$ is $\alpha$-H\"older continuous, then $\dim f(A)\leq \frac{1}{\alpha} \dim A.$
\end{fact}

\section{H\"older restrictions}

The goal of this section is to prove Theorem~\ref{t:Holder}. First we need some preparation.

\begin{definition} A function $g\colon [0,1]\to \mathbb{R}^2$ is called \emph{dimension doubling} if
$$\dim g(A)=2\dim A \quad \textrm{for all } A\subset [0,1].$$
\end{definition}

\begin{theorem}[Kaufman, \cite{K1}, see also \cite{MP}] The two-dimensional Brownian motion is almost surely dimension doubling.
\end{theorem}

The following theorem follows from \cite[Lemma 2]{H} together with the fact that the closed range of the stable subordinator with parameter
$\frac 12$ coincides with the zero set of a linear Brownian motion. For a more direct reference see \cite{K2}.

\begin{theorem} \label{t:zero} Let $A\subset [0,1]$ be a compact set with $\dim A>\frac 12$ and let $\mathcal{Z}$ be the zero set
of a linear Brownian motion. Then $\dim(A\cap \mathcal{Z})>0$ with positive probability.
\end{theorem}

\begin{lemma}[Key Lemma] \label{l:key} Let $\{W(t)\colon 0\leq t\leq1\}$ be a linear Brownian motion. Assume that $\alpha>\frac 12$ and $f\colon [0,1]\to \mathbb{R}$ is a continuous function such that $(f,W)$ is almost surely dimension doubling. Then there is no set $A\subset[0,1]$ such that $\dim A>\frac12$
and $f$ is $\alpha$-H\"older continuous on $A$.
\end{lemma}

\begin{proof} Assume to the contrary that there is a set $A\subset [0,1]$ such that $\dim A>\frac12$
and $f$ is $\alpha$-H\"older continuous on $A$. As $f$ is still $\alpha$-H\"older
continuous on the closure of $A$, we may assume that $A$ itself is closed. Let $\mathcal{Z}$ be the zero set of $W$,
then Theorem~\ref{t:zero} implies that $\dim(A\cap \mathcal{Z})>0$ with
positive probability. Then the $\alpha$-H\"older continuity of $f|_{A}$ and Fact~\ref{f:H} imply that, with positive probability,
\begin{align*} \dim (f,W)(A\cap \mathcal{Z})&=\dim (f(A\cap \mathcal{Z})\times \{0\})=\dim f(A\cap \mathcal{Z})\\
&\leq \frac1{\alpha}\, \dim(A\cap \mathcal{Z}) < 2 \dim(A\cap \mathcal{Z}),
\end{align*}
which contradicts the fact that $(f,W)$ is almost surely dimension doubling.
\end{proof}

\begin{proof}[Proof of Theorem~\ref{t:Holder}]
Let $\{W(t)\colon 0\leq t\leq 1\}$ be a linear Brownian motion which is independent of $B$. By Kaufman's dimension doubling theorem $(B,W)$ is
dimension doubling with probability one, thus applying Lemma~\ref{l:key} for an almost sure path of $B$ finishes the proof.  \end{proof}

\section{Restrictions of bounded variation}

We need the following lemma, which may be obtained by a slight modification of \cite[Lemma~4.1]{ABPR}.
For the reader's convenience we outline the proof.

\begin{lemma} \label{l:BV}
Let $\alpha, \beta >0$. Assume that $A\subset [0,1]$ and the function $f\colon A \to \mathbb{R}$ has finite $\beta$-variation. Then there are sets
$A_n\subset A$ such that for any $n\in \mathbb{N}^+$
$$f|_{A_n} \textrm{ is $\alpha$-H\"older continuous and } \dim \left(A\setminus \bigcup_{n=1}^{\infty} A_n\right)\leq \alpha \beta.$$
\end{lemma}

\begin{proof} For all $n\in \mathbb{N}^+$ let
$$A_n=\{x\in A: |f(x+t)-f(x)|\leq 2t^{\alpha} \textrm{ for all } t\in [0, 1/n]\cap (A-x)\}.$$
As $A$ is bounded, $f|_{A_n}$ is $\alpha$-H\"older continuous for all $n\in \mathbb{N}^+$. Let
$$D=\left\{x\in A: \limsup_{t\to 0+} |f(x+t)-f(x)|t^{-\alpha}>1\right\}.$$
Clearly $A\setminus \bigcup_{n=1}^{\infty} A_n\subset D$, so it is enough to prove that $\dim D\leq \alpha \beta$. Let us fix $\delta>0$ arbitrarily.
Then for all $x\in D$ there is a $0<t_x<\delta$ such that
\begin{equation} \label{eq:tx} |f(x+t_x)-f(x)|\geq t^{\alpha}_x.
\end{equation}
Define $I_x=[x-t_x,x+t_x]$ for all $x\in D$.
By Besicovitch's covering theorem (see \cite[Thm.~2.7]{M}) there is a number $p\in \mathbb{N}^+$ not depending on $\delta$ and countable sets
$S_i\subset D$ ($i\in \{1,\dots, p\}$) such that
\begin{equation} \label{eq:D} D\subset \bigcup_{i=1}^{p} \bigcup_{x\in S_i} I_x \textrm{ and } I_x\cap I_y=\emptyset \textrm{ for all } x,y\in S_i,~x\neq y.
\end{equation}
Applying \eqref{eq:tx} and \eqref{eq:D} implies that for all $i\in \{1,\dots,p\}$
\begin{equation} \label{eq:Var} \sum_{x\in S_i} |I_x|^{\alpha \beta} =2^{\alpha \beta} \sum_{x\in S_i}t_x^{\alpha \beta}\leq 2^{\alpha \beta} \sum_{x\in S_i} |f(x+t_x)-f(x)|^{\beta}  \leq 2^{\alpha \beta} \Var^{\beta} (f).
\end{equation}
Equations~\eqref{eq:D} and \eqref{eq:Var} imply that
$$\mathcal{H}_{\delta}^{\alpha \beta}(D)\leq \sum_{i=1}^{p} \sum_{x\in S_i} |I_x|^{\alpha \beta}\leq p2^{\alpha \beta}\Var^{\beta} (f).$$
As $\Var^{\beta} (f)<\infty$ and $\delta>0$ was arbitrary, we obtain that $\mathcal{H}^{\alpha \beta}(D)<\infty$.
Hence $\dim D\leq \alpha \beta$, and the proof is complete.
\end{proof}

\begin{proof}[Proof of Theorem~\ref{t:BV}]
Assume to the contrary that for some $\varepsilon>0$ there is a random set $A\subset [0,1]$ such that, with positive probability,
$\dim A\geq \beta/2+2\varepsilon$ and $B|_{A}$ has finite $\beta$-variation. Let $\alpha=1/2 +\varepsilon/\beta>1/2$.
Applying Lemma~\ref{l:BV} we obtain that there are sets $A_n\subset A$ such that $B|_{A_n}$ is $\alpha$-H\"older continuous for every $n\in \mathbb{N}^+$ and
\begin{equation} \label{eq:alpha}  \dim \left(A\setminus \bigcup_{n=1}^{\infty} A_n\right)\leq \alpha \beta =\frac{\beta}{2}+\varepsilon.
\end{equation}
As $\alpha>1/2$ and $B|_{A_n}$ are $\alpha$-H\"older continuous,
Theorem~\ref{t:Holder} implies that almost surely $\dim A_n\leq 1/2$ for all $n\in \mathbb{N}^+$,
therefore \eqref{eq:alpha} and the countable stability of the Hausdorff dimension yield that
$\dim A\leq \beta/2+\varepsilon$ almost surely. This is a contradiction, and the proof is complete.
\end{proof}

Theorems~\ref{t:L} and \ref{t:opt} (with $\alpha=\frac 12$) imply that Theorem~\ref{t:BV} is sharp for all $\beta$.

\begin{theorem}[L\'evy's modulus of continuity, \cite{L}, see also \cite{MP}] \label{t:L} For the linear Brownian motion $\{ B(t) \colon 0\leq t\leq 1\}$, almost surely,
$$\limsup_{h\to 0+} \sup_{0\leq t\leq 1-h} \frac{|B(t+h)-B(t)|}{\sqrt{2h \log(1/h)}}=1.$$
\end{theorem}

\begin{theorem}\label{t:opt} Let $0<\alpha\leq 1$ and $0<\beta \leq \frac {1}{\alpha}$ be fixed. Then there is a compact set $A\subset [0,1]$ such that $\dim A=\alpha \beta$ and if $f\colon [0,1]\to \mathbb{R}$ is a function and $c\in (0,\infty)$ such that for all $x,y\in [0,1]$
\begin{equation} \label{eq:mod} |f(x)-f(y)|\leq c|x-y|^{\alpha}\log \frac{1}{|x-y|}, \end{equation}
then $f|_{A}$ has finite $\beta$-variation.
\end{theorem}

\begin{proof} First we construct $A$. For all $n\in \mathbb{N}$ let
$$\gamma_n=2^{-n/(\alpha \beta)}(n+1)^{-(\beta+2)/\beta}.$$
We define intervals $I_{i_1\dots i_n}\subset [0,1]$ for all $n\in \mathbb{N}$ and $\{i_1,\dots, i_n\}\in \{0,1\}^{n}$  by induction.
We use the convention $\{0,1\}^{0}=\{\emptyset\}$. Let $I_{\emptyset}=[0,1]$, and if the interval
$I_{i_1\dots i_n}=[u,v]$ is already defined then let
$$I_{i_1\dots i_n 0}=[u,u+\gamma_{n+1}] \textrm{ and } I_{i_1\dots i_n 1}=[v-\gamma_{n+1},v].$$
Let
$$A=\bigcap_{n=0}^{\infty} \bigcup_{(i_1,\dots,i_n) \in \{0,1\}^n} I_{i_1\dots i_n}.$$
Assume that $f\colon [0,1]\to \mathbb{R}$ is a function satisfying \eqref{eq:mod}. Now we prove that $\Var^{\beta} (f|_{A}) <\infty$. As $\diam I_{i_1\dots i_n}=\gamma_n$,
the definition of $\gamma_n$ and \eqref{eq:mod} imply that for all $n\in \mathbb{N}$ and $(i_1,\dots, i_n)\in \{0,1\}^n$ we have
\begin{equation} \label{eq:beta}
(\diam f(I_{i_1\dots i_n}))^\beta \leq (c \gamma_n^{\alpha} \log \gamma_n^{-1})^{\beta}\leq  c_{\alpha,\beta}2^{-n}(n+1)^{-2},
\end{equation}
where $c_{\alpha,\beta}\in (0,\infty)$ is a constant depending on $\alpha,\beta$ and $c$ only. For all
$x,y\in A$ let $n(x,y)$ be the maximal number $n$ such that $x,y\in I_{i_1\dots i_n}$ for some $(i_1,\dots,i_n)\in \{0,1\}^n$. If $\{x_i\}_{i=0}^{k}$ is a monotone sequence in $A$ and $n\in \mathbb{N}$
then the number of $i\in \{1,\dots,k\}$ such that $n(x_{i-1},x_{i})=n$ is at most $2^n$. Therefore \eqref{eq:beta} implies that
$$\Var^{\beta} (f|_{A}) \leq \sum_{n=0}^{\infty} 2^{n} \left(c_{\alpha,\beta} 2^{-n}(n+1)^{-2}\right)=\sum_{n=1}^{\infty} c_{\alpha,\beta} n^{-2}<\infty.$$
Finally, we prove that $\dim A=\alpha \beta$. The upper bound $\dim A\leq \alpha \beta$ is obvious, thus we show only the lower bound. In the construction of $A$ each $(n-1)$st level interval $I_{i_1\dots i_{n-1}}$ contains $m_n=2$ $n$th level intervals $I_{i_1\dots i_{n-1}i}$, which are separated by gaps of $\varepsilon_n=\gamma_{n-1}-2\gamma_n$. The definition of $\gamma_n$ yields that $0<\varepsilon_{n+1}<\varepsilon_n$ for all $n\in \mathbb{N}^+$
and $\varepsilon_n=2^{-n/(\alpha \beta)+o(n)}$. Applying \cite[Example~4.6]{F} we obtain that
$$\dim A\geq \liminf_{n\to \infty} \frac{\log(m_1\cdots m_{n-1})}{-\log(m_n \varepsilon_n)}=\alpha \beta,$$
and the proof is complete.
\end{proof}

\section*{Acknowledgment} The first author was supported by the
Hungarian Scientific Research Fund grant no.~104178. We thank Russell Lyons and Nicolas Curien for help with the abstract.


\begin{thebibliography}{99}

\bibitem{ABP} O. Angel, R. Balka, Y. Peres, Increasing subsequences of random walks, submitted, arXiv:1407.2860.

\bibitem{ABPR} T. Antunovi\'c, K. Burdzy, Y. Peres, J. Ruscher,
Isolated zeros for Brownian motion with variable drift, \textit{Electron. J. Probab.} \textbf{16} (2011), no. 65, 1793--1814.

\bibitem{E} M. Elekes, Hausdorff measures of different dimensions are isomorphic under the Continuum Hypothesis,
\textit{Real Anal. Exchange} \textbf{30} (2004), no. 2, 605--616.

\bibitem{F} K. Falconer, \textit{Fractal geometry: mathematical
foundations and applications}, Second Edition, John Wiley \& Sons, 2003.

\bibitem{H} J. Hawkes, On the Hausdorff dimension of the intersection of the range of a stable process with a Borel set,
\textit{Z. Wahrscheinlichkeit} \textbf{19} (1971), 90--102.

\bibitem{KK} J.-P. Kahane, Y. Katznelson, Restrictions of continuous functions, \textit{Israel J. Math.} \textbf{174} (2009), 269--284.

\bibitem{K2} R. Kaufman, Measures of Hausdorff-type, and Brownian motion, \textit{Mathematika} \textbf{19} (1972), 115--119.

\bibitem{K1} R. Kaufman, Une propri\'et\'e m\'etrique du mouvement brownien, \textit{C. R. Acad. Sci. Paris} \textbf{268} (1969), 727--728.

\bibitem{L} P. L\'evy, \textit{Th\'eorie de l'addition des variables al\'eatoires}, Gauthier-Villars, Paris, 1937.

\bibitem{Ma} A. M\'ath\'e, Measurable functions are of bounded variation on a set of
Hausdorff dimension $\frac{1}{2}$, \textit{Bull. London Math. Soc.} \textbf{45} (2013), 580--594.

\bibitem{M} P. Mattila, \textit{Geometry of sets and measures in Euclidean spaces}, Cambridge Studies in Advanced Mathematics No.~44,
Cambridge University Press, 1995.

\bibitem{MP} P. M\"orters, Y. Peres, \textit{Brownian Motion}, with an appendix by
Oded Schramm and Wendelin Werner, Cambridge University Press, 2010.

\end{thebibliography}
\end{document}